\documentclass[a4paper,11pt]{article}
\usepackage{latexsym}
\usepackage{amssymb}
\usepackage{bbm}
\usepackage{enumitem}
\usepackage{amsfonts}
\usepackage{amsmath}
\usepackage{amsthm}
\usepackage{cases}
\usepackage[pdftex]{graphicx} 
\usepackage[paper=a4paper,left=30mm,right=20mm,top=25mm,bottom=30mm]{geometry}

\newenvironment{@abssec}[1]{%
    \if@twocolumn

      \section*{#1}%
    \else

      \vspace{.05in}\footnotesize
      \parindent .2in
 {\upshape\bfseries #1. }\ignorespaces
    \fi}

    {\if@twocolumn\else\par\vspace{.1in}\fi}

\newenvironment{keywords}{\begin{@abssec}{\keywordsname}}{\end{@abssec}}

\newenvironment{AMS}{\begin{@abssec}{\AMSname}}{\end{@abssec}}

\newcommand\keywordsname{Key words}
\newcommand\AMSname{AMS subject classifications}
\newcommand\AMname{AMS subject classification}
\newtheorem{theorem}{Theorem}
 \newtheorem{lemma}[theorem]{Lemma}

\title{Patterns with prescribed numbers of critical points on topological tori
\footnotetext[3]{This research was partially supported by the Grant-in-Aid
for Scientific Research (B) ($\sharp$ 18H01126) of
Japan Society for the Promotion of Science.}\footnotemark[3]}

\author{Putri Zahra Kamalia\footnote{Corresponding author.} \footnote{Research Center for Pure and Applied Mathematics, Graduate School of Information Sciences, Tohoku
University, Sendai, 980-8579, Japan ({\tt putrizahrakamalia@gmail.com}${}^*$, {\tt  sigersak@tohoku.ac.jp}).}\  \  and \  Shigeru Sakaguchi\footnotemark[2]} 

\date{}
\begin{document}

\maketitle

\begin{abstract}
We study the existence of critical points of stable stationary solutions to reaction-diffusion problems on topological tori. Stable nonconstant stationary solutions are often called patterns. 
We construct topological tori and patterns with prescribed numbers of critical points whose locations are explicit.
\end{abstract}



\section{Introduction}
Let $M$ be a topological torus equipped with a Riemannian metric $g$. For $u =u(x,t)$ on $M$, we consider the following reaction-diffusion problem
\begin{equation}\label{eq1}
\partial_t u= \Delta_{g} u + f(u)\quad \mbox{ in } M \times (0,\infty),
\end{equation}
where $f \in C^1(\mathbb{R})$  is a function of $u$  and $\Delta_g$  denotes  the Laplace-Beltrami operator on $M$,
\begin{equation}\label{gradu}
\Delta_{g} u = \textrm{div} (\nabla_{g} u)= \sum_{i=1}^2 \frac{1}{\sqrt{|g|}} \frac{\partial}{\partial x^i}  \left(\sqrt{|g|} (\nabla_{g} u)^i \right).
\end{equation}
A stationary solution $U$  of \eqref{eq1} is said to be stable in the sense of Lyapunov if for each $\epsilon > 0$, there exists a $\delta > 0$ such that, for every initial data $u_0$ with
$\Vert {u_0-U}\Vert_{\infty} < \delta$ we have
$
\Vert u(\cdot,t)-U\Vert_{\infty} < \epsilon\ \mbox{ for every } t > 0.
$
Throughout, we will refer to stable nonconstant stationary solutions as patterns. 

The existence and nonexistence of patterns on surfaces of revolution and more general compact $d$-dimensional Riemannian manifolds have been studied in \cite{bandle, farina, jimbo, nascimento, punzo, rubinstein, sonego}. Among these, in \cite[Theorem 2]{jimbo} Jimbo introduced manifolds and nonlinearities $f$ having complex patterns whose construction is analogous to that in \cite{matano}, and in \cite{bandle} Bandle, Punzo and Tesei constructed a class of surfaces of revolution with non-empty boundary and nonlinear terms $f$ having patterns with the Neumann boundary condition by solving some ordinary differential equations with the aid of an idea introduced by Yanagida in \cite{yanagida} to  construct the nonlinear terms $f$.  

In this paper, we study the stability of patterns of \eqref{eq1} where $M$ is a small perturbation $T^2_\epsilon$ of the standard tori $T^2$. Our purpose is to find complex patterns of \eqref{eq1} on topological tori with exact numbers of critical points. 

In our previous work \cite{kamal}, we constructed topological tori $M$ together with patterns on $M$ having at least $4n$ critical points for sufficiently large $n$.  In the beginning, we slightly perturb the surfaces of revolution $D$ in \cite{bandle} in such a way that each center curve of each new surface  $M_\kappa$ is just a circular arc with sufficiently small curvature $\kappa$, where $M_0=D$.
  The perturbed patterns  correspond to  the following reaction-diffusion problem with the Neumann boundary condition
\begin{equation}\label{rdn}
\begin{cases}
\ \ \partial_t u= \Delta_g u + f(u) &\mbox{ in } M_\kappa \times (0,\infty),\\
\ \ \dfrac{\partial u}{\partial \nu}=0 &\mbox{ on } \partial M_\kappa \times (0,\infty),
\end{cases}
\end{equation}
where $\nu$ denotes the outward unit normal vector to the boundary $\partial M_\kappa$.  With the aid of the implicit function theorem,  the patterns of \eqref{rdn} exist on $M_\kappa$ with no critical point in the interior of $M_\kappa$. Furthermore, we attach a sufficiently large even number $2n$ of copies of $M_\kappa$ together with the perturbed pattern on $M_\kappa$ to each other in such a way that the center curve of the new closed surface $M$ is just a whole circle with curvature $\kappa$.  As a consequence, the stationary solution of \eqref{eq1} on $M$ is constructed as  a symmetric function on $M$. The stability of the constructed solution on $M$ having at least $4n$ critical points follows from  the symmetry coming from the construction. We can only show that there exists at least two critical point on each boundary  $\partial M_\kappa$ in $M$ and hence the constructed pattern has at least $4n$ critical points.

 The objective of this paper is to show that patterns of \eqref{eq1} exist on topological tori  with an exact number $4n$ of critical points by introducing  another simple way to construct  topological tori together with its patterns. In \cite{kamal}  we start with a small piece of topological tori $M_\kappa$ and we then arrive at the whole tori $M$ by attaching a sufficiently large even number of copies of $M_\kappa$. On the other hand,  in this paper we directly perturb the standard tori $T^2$ to obtain new topological tori $T^2_\epsilon$ with a small parameter $\epsilon$, where $T^2_0=T^2$. Topologically, $T^2_\epsilon$  is the same as $M$ constructed in \cite{kamal}, but geometrically, they are different from each other. The reason why  we are able to give an exact number of critical points on $T^2_\epsilon$ is simply because we deal with the explicit perturbation $T^2_\epsilon$ of the standard torus $T^2$ in this paper. 
 


We start with considering the upper half of a standard torus $T^2$. By some adjustment, using \cite[Theorem 4.1, p. 41]{bandle} yields patterns of problem \eqref{rdn} where $M_\kappa$ is replaced by  the upper half of $T^2$. Subsequently, we prove that patterns of \eqref{eq1} exist on $T^2$.  Next, we slightly perturb $T^2$  by simply changing the radius of the tube from a constant into a periodic function to obtain  new topological tori $T^2_\epsilon$ with a small parameter $\epsilon$. Then the patterns of \eqref{eq1} on $T^2$ together with the implicit function theorem yield patterns of \eqref{eq1} on $T^2_\epsilon$.  Moreover,  the stability of each pattern enables us to examine the critical points of the pattern in $T^2_\epsilon$.  We summarize our result in the following theorem. 
\begin{theorem}\label{mt}
	There exist a nonlinearity $f$ and a number $N \in \mathbb{N}$  such that, for each $n \ge N$,  a perturbation $M$ of a standard torus $T^2$ together with a pattern $U$ of \eqref{eq1}, is constructed in such a way that $U$ has exactly $4n$ critical points.
\end{theorem}

We organize the paper as follows. In section 2, we introduce patterns on standard tori with the aid of a result in \cite{bandle}.  In section 3, we start with the construction of perturbed tori and give a proof of Theorem \ref{mt}. We also mention the exact locations of the critical points of the patterns.

\setcounter{equation}{0}
\setcounter{theorem}{0}
\section{Standard tori}
Let  $T^2$ be a standard torus properly embedded in $\mathbb R^3$ and parameterized by 
\begin{equation}  \label{param}
\begin{cases}
\ \ x_1=(R+r \cos\varphi ) \cos \theta , \\
\ \ x_2=(R+r \cos\varphi ) \sin \theta,  \hspace{1cm} ( ( \varphi,\theta) \in S^1 \times S^1)\\ 
\ \ x_3=r \sin \varphi,
\end{cases}
\end{equation}
where $R, r$ are constants and $R>r>0$. 
Set $x^1=\varphi, x^2=\theta$.  Then $T^2$ is a 2-dimensional Riemannian manifold with metric $ds^2$ and area element $d\sigma$ given by
\begin{equation*}
\begin{cases}
\ \ ds^2=\sum_{i,j=1}^{2} g_{ij}dx^idx^j=r^2 d\varphi^2+ (R+r\cos \varphi)^2 d\theta^2,\\
\ \ d\sigma=\sqrt{|g|} d\varphi d\theta=r(R+r\cos \varphi)   d\varphi d\theta.
\end{cases}
\end{equation*}
The Riemannian gradient $\nabla_g u$ of $u$ with respect to $g$ on $T^2$ is  given by
 \begin{equation*}
 \nabla_{g} u =
 \begin{pmatrix}
\displaystyle \dfrac{1}{r^2} \partial_\varphi u \\[0.5cm]
 \dfrac{1}{(R+r\cos \varphi)^2}  \partial_\theta u
 \end{pmatrix},
 \end{equation*}
and the Laplace-Beltrami operator $\Delta_g$ on $T^2$ is expressed as
 \begin{equation}\label{Laplace operator for torus}
 \Delta_{g} u= \frac{1}{r^2} u_{\varphi\varphi}+\frac{1}{(R+r\cos \varphi)^2} u_{\theta\theta} - \frac{\sin \varphi}{r(R+r \cos \varphi)} u_{\varphi}.
 \end{equation}
 Geometrically,  $T^2$  is a surface of revolution obtained by revolving a circle with radius $r$ with center $(R,0,0)$ in $x_1x_3$ plane about the $x_3$-axis. 
 
 For a surface of revolution $D$,  in \cite{bandle} Bandle et al. studied the existence and nonexistence of patterns of the reaction-diffusion problem on $D$ with the Neumann boundary condition
\begin{equation}\label{rdnBandle}
\begin{cases}
\ \ \partial_t u= \Delta_g u + f(u) &\mbox{ in } D \times (0,\infty),\\
\ \ \dfrac{\partial u}{\partial \nu}=0 &\mbox{ on } \partial D \times (0,\infty),
\end{cases}
\end{equation}
where $\nu$ denotes the outward unit normal vector to $\partial D$. 

 Consider  the eigenvalue problem linearized at a stationary solution $U$ of  \eqref{eq1} (or \eqref{rdnBandle}):
\begin{equation}\label{egvals}
\begin{cases}
\ \ \Delta_{g} \phi+f^\prime(U) \phi=-\lambda \phi \mbox{ in } M\ (\mbox{or } D),\\
\ \ \left(\dfrac{\partial \phi}{\partial \nu}=0 \mbox{ on } \partial D \mbox{ for problem \eqref{rdnBandle}}\right).
\end{cases}
\end{equation}
The Rayleigh quotient of problem \eqref{egvals} can be taken on in terms of the principal eigenvalue
\begin{equation} \label{ev}
\lambda_1=  \inf_{\substack{\phi \neq 0\\\phi \in H^1(M) (\mbox{or }H^1(D))}} \dfrac{ \int\limits_{M (\mbox{or } D)} \left( |\nabla_{g} \phi|^2 - f^\prime (U)\phi^2 \right) d\sigma}{\int\limits _{ M (\mbox{or } D)} \phi^2 d\sigma}.
\end{equation}
If $\phi_1$ is the normalized eigenfunction that corresponds to the principal eigenvalue $\lambda_1$, then it satisfies
\begin{align}
\begin{cases}
\ \ \Delta_{g} \phi_1+f^\prime(U) \phi_1=-\lambda_1\phi_1\ \mbox{ in }\, M (\mbox{or } D),\\ 
\ \ \left(\dfrac{\partial \phi_1}{\partial \nu}=0 \mbox{ on } \partial D \mbox{ for problem \eqref{rdnBandle}}\right),\\
\ \ \|\phi_1\|_{L^2(M) (\mbox{or } L^2(D))}=1, \quad \phi_1>0\ \mbox{ in }\, M (\mbox{or } D),\\
\ \ \lambda_1=\int\limits_{M (\mbox{or } D)} \left( |\nabla_{g} \phi_1|^2 - f^\prime (U)\phi_1^2 \right) d\sigma
\end{cases}
\end{align}
Note that the normalized eigenfunction that corresponds to the principal eigenvalue $\lambda_1$ is uniquely determined.
It is well known that the sign of the principal eigenvalue determines the stability of a stationary solution, with the following stability criterion (see \cite{eigen} for instance):
\begin{itemize}
	\item $U$ is stable, if $\lambda_1 > 0$,
	\item $U$ is unstable, if $\lambda_1 < 0$,
	\item The stability of $U$ is undetermined, if $\lambda_1 =0$.
\end{itemize}

In particular, we suppose that $D$ is parameterized by 
\begin{equation}
\begin{cases}
\ \ x_1=\psi(\rho) \cos \theta, \\
\ \ x_2=\psi(\rho) \sin \theta,   \qquad \left( (\rho, \theta) \in [0,L]\times [0,2\pi) \right)\\ 
\ \ x_3=\chi(\rho),
\end{cases}
\end{equation}
where $L > 0$ is a constant and $\psi, \chi \in C^3([0,L])$ satisfy that 
$$
\psi>0\ \mbox{ and }\ (\psi^\prime )^2+(\chi^\prime)^2=1\ \mbox{ on }\ [0,L]. 
$$
Recall  a theorem of \cite{bandle} which states
\begin{theorem}{\rm(\cite[Theorem 4.1, p. 41]{bandle})}
 \label{main1}
Suppose that for some $\rho_0 \in (0,L)$
\begin{equation}\label{stas}
\left(\frac{\psi^\prime}{\psi}\right)^\prime> 0 \text{ at } \rho=\rho_0.
\end{equation}
Then, there exists $f\in C^1 (\mathbb{R})$ such that problem \eqref{rdnBandle} admits a pattern $Z=Z(\rho)$, where the principal eigenvalue $\lambda_1$ of the eigenvalue problem linearized  at $Z$
is positive.
\end{theorem}
The pattern $Z=Z(\rho)$ of \eqref{rdnBandle} on $D$ is a positive function of one variable $\rho$ (see  \cite[(4.8),(4.9), p.43 and Proposition 3.2., p. 39]{bandle}).
The nonlinear term $f=f(Z)$ changes its sign and is defined by \cite[(4.10), p. 43]{bandle} in such a way that for any $\rho \in(0,L)$
\begin{equation}\label{f}
f[Z(\rho)]= -\dfrac{(\psi Z^\prime)^\prime}{\psi}(\rho).
\end{equation}

 Let us set 
$$
\rho=r\varphi, \ L=r\pi,\  \psi(\rho)=R+r\cos(r^{-1}\rho)\  \mbox{ and }\chi(\rho)= r \sin(r^{-1}\rho).
$$
Then $D$ corresponds to $T^+,$ the upper half of the standard torus $T^2$.   
\begin{figure}[h!]
	\centering
	\includegraphics[width=0.5\textwidth]{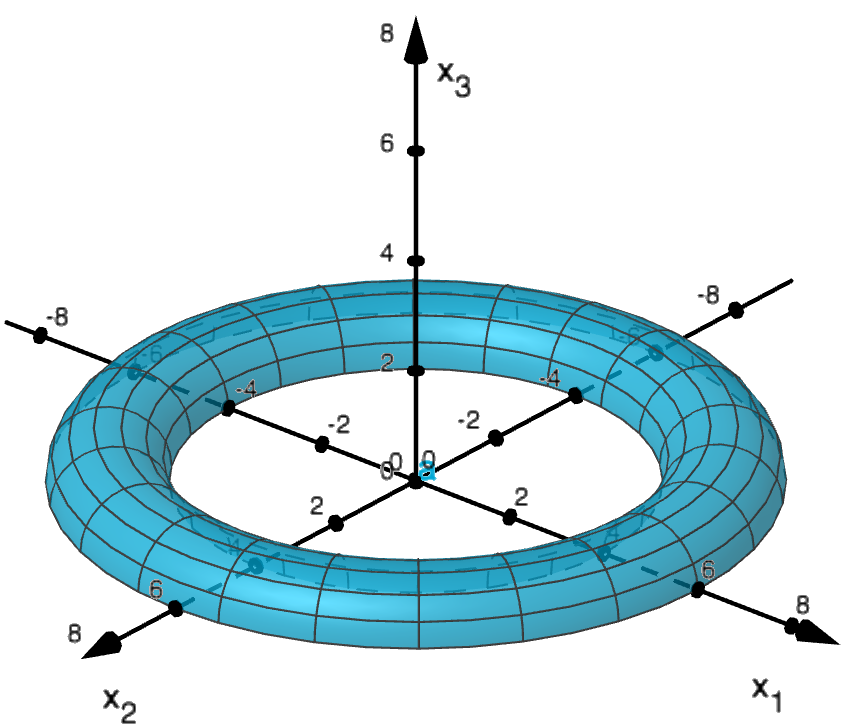}
	\caption{the upper half $T^+$ of $T^2$ with $R=5$ and $r=1$}
	\label{tree}
\end{figure}
If $\varphi \in \left(0, \pi\right)$ is sufficiently close to $\pi$, then it follows that
$$
\left(\dfrac{\psi^\prime}{\psi}\right)^\prime = -\dfrac{r+R \cos \varphi}{r(R +r\cos \varphi)^2} > 0,
$$  
which  guarantees  \eqref{stas}.  Hence,  by Theorem \ref{main1}, a pattern $Z(\rho)=U^*(\varphi)$ of  \eqref{rdnBandle}  exists on $T^+$ for $f$ given by \eqref{f}. Moreover, we have from \cite[proof of Lemma 4.4, pp. 43--44]{bandle} that
\begin{enumerate}[label=(\alph*)]
	\item   $U^*(\varphi)$  is positive in $(0, \pi]$;
	\item \label{inc} $U^*(\varphi)$ is strictly increasing in $(0,\pi)$; 
	\item  $U^*(\varphi)$ has no critical points in the interior of $T^+$.
\end{enumerate}
We add one remark in order to apply Theorem \ref{main1}(\cite[Theorem 4.1, p. 41]{bandle}) to our surface of revolution $T^+$. Bandle et al. assume another condition \cite[(2.13), p. 37]{bandle} which is not satisfied for our $T^+$, but the following condition holds true for $T^+$
$$
\frac {\partial}{\partial\nu}=\frac {\partial}{\partial\rho} \mbox{ on } \{\varphi=\pi\}\ \mbox{ and }\ \frac {\partial}{\partial\nu}=-\frac {\partial}{\partial\rho} \mbox{ on } \{\varphi=0\},
$$ 
as they mention it at \cite[just after (2.13), p. 37]{bandle} and use it at \cite[line 14, p. 44]{bandle}. Thus, we are able to apply Theorem \ref{main1} to our $T^+$.

Consider  the eigenvalue problem \eqref{egvals} linearized at the stationary solution $U^*$ of   \eqref{rdnBandle} for $D=T^+$, and let $\phi^*_1$ be 
the normalized eigenfunction that corresponds to the principal eigenvalue $\lambda_1$.  Since $U^*$ is a function of one variable $\varphi$ and the normalized eigenfunction is uniquely determined, we see that $\phi^*_1$ is also a function of one variable $\varphi$.

\subsection{Patterns on standard tori}

Let $T^-$ be the lower half of $T^2$. Then 
$$
\partial T^+ =\partial T^-= T^2\cap\{x_3=0\}\ \mbox{ and }\ T^2=T^+\cup\partial T^+\cup T^-.  
$$
Moreover, $\partial T^+$ consists of  the two horizontal circles $C_{max}, C_{min}$ whose  radii are the maximum $R+r$ and the minimum $R-r$,  respectively, and hence
$$
T^2\cap\{x_3=0\} = C_{max}\cup C_{min}.
$$

\begin{theorem}\label{t2}
	For the standard torus $M=T^2$, there exists a nonlinearity $f$ together with a pattern $U$ of \eqref{eq1} with $\lambda_1 > 0$ such that the set of  critical points of $U$ equals $C_{max}\cup C_{min}$.
\end{theorem}
\begin{proof} Since $U^*$ satisfies the Neumann boundary condition on $\partial T^+$, we can define a stationary solution $U=U(\varphi)$ of \eqref{eq1} for  $M=T^2$ by
\begin{equation}\label{sym}
U(\varphi)=\left\{ \begin{array}{ll}
    U^*(\varphi) &\mbox{ if }\ \varphi \in [0,\pi],\\
    U^*(2\pi-\varphi)  & \mbox{ if }\ \varphi \in (\pi,2\pi).
  \end{array} \right.
  \end{equation}
Observe that
\begin{enumerate}[label=(\roman*),font=\upshape]
	\item \label{ss}  $U$ is symmetric with respect to each component of  $T^2\cap\{x_3=0\}$;
	\item $U(\varphi)$  is positive in $(0,2\pi)$;
	\item $U(\varphi)$ is strictly increasing in $(0,\pi)$ and strictly decreasing  in $(\pi,2\pi)$;
	\item \label{inf}  The set of  critical points of $U$ equals $T^2\cap\{x_3=0\}$, and $U$ achieves its positive maximum on $C_{min}$ and zero minimum on $C_{max}$. 
\end{enumerate}
Thus, it suffices to prove
\begin{lemma}\label{stab}
	The principal eigenvalue $\lambda_1$ of the eigenvalue problem linearized  at $U$ is positive and hence $U$ is stable.
\end{lemma}
\noindent
{\it Proof. } Consider the eigenvalue problem \eqref{egvals} linearized at the stationary solution $U$ of \eqref{eq1} for $M=T^2$. We may define the normalized eigenfunction $\phi_1=\phi_1(\varphi)$, that corresponds to the same principal eigenvalue $\lambda_1 (>0)$ as for $T^+$, by
\begin{equation}\label{sym_eigenfunction}
\phi_1(\varphi)=\left\{ \begin{array}{ll}
    \frac 1{\sqrt{2}}\phi_1^*(\varphi) &\mbox{ if }\ \varphi \in [0,\pi],\\
     \frac 1{\sqrt{2}}\phi_1^*(2\pi-\varphi)  & \mbox{ if }\ \varphi \in (\pi,2\pi).
  \end{array} \right.
  \end{equation}
  By the uniqueness of the normalized eigenfunction  that corresponds to the principal eigenvalue, this $\phi_1$ is exactly the eigenfunction we want and hence the principal eigenvalue $\lambda_1$ is positive.
\end{proof}

\setcounter{equation}{0}
\setcounter{theorem}{0}

\section{Standard tori with perturbation}
In this section we will slightly perturb $T^2$ and analyze its stability along with the existence of critical points.
 Let $T^2_\epsilon$ denote a perturbation of  $T^2$  where $T^2_0=T^2$. To be precise, $T^2_\epsilon$ is parameterized by
 \begin{equation} \label{parametrization of perturbed tori}
\begin{cases}
\ \ x_1=(R+r_\epsilon(\theta) \cos\varphi ) \cos \theta , \\
\ \ x_2=(R+r_\epsilon(\theta) \cos\varphi ) \sin \theta,  \hspace{1cm} ( ( \varphi,\theta) \in I:=S^1 \times S^1)\\ 
\ \ x_3=r_\epsilon(\theta) \sin \varphi,
\end{cases}
\end{equation}
where  $n\in \mathbb{N}$, $r_\epsilon(\theta)=r+\epsilon \sin (n\theta)$,  and the constants $R, r, \varepsilon$ satisfy $R> r+|\epsilon|=\max\limits_{\theta \in S^1}  r_\epsilon(\theta)$.
\begin{figure}[h!]
	\centering
	\includegraphics[width=0.5\textwidth]{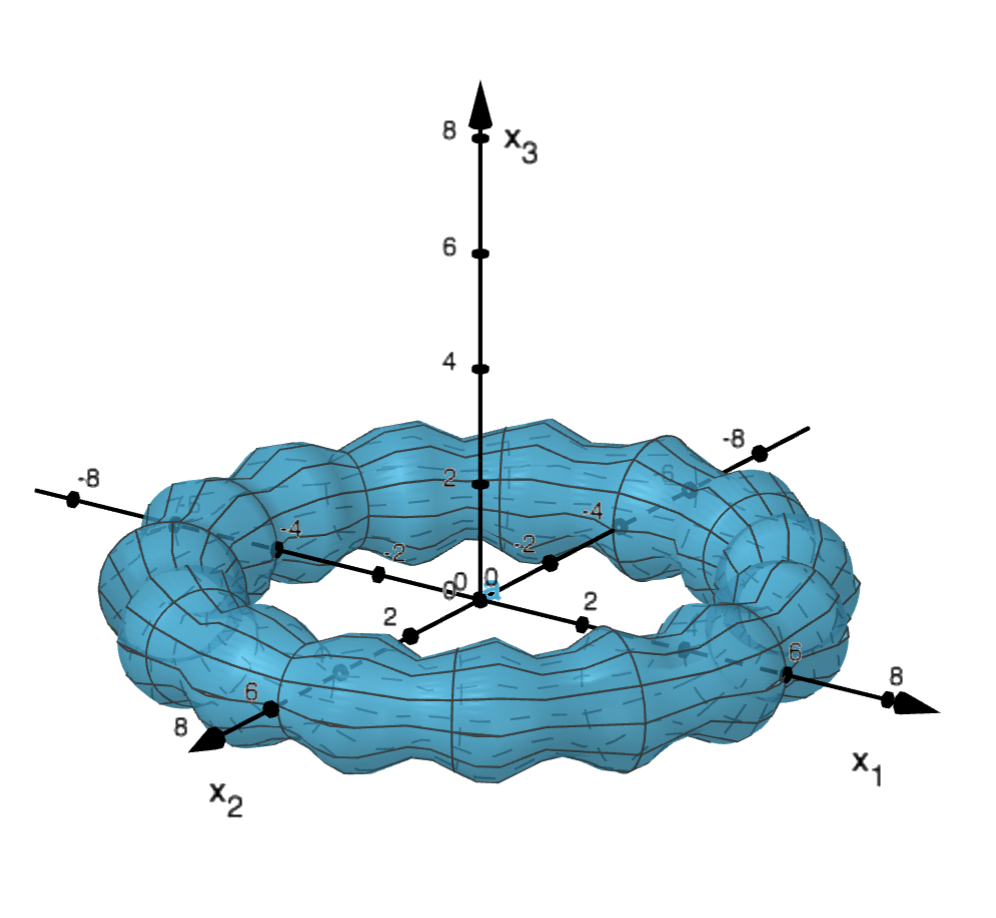}
	\caption{$T^2_\epsilon$ with $R=5, \epsilon=0.2, \text{ and }n=15$}
\end{figure}
Set $x^1=\varphi, x^2=\theta$.  Then
the corresponding Riemannian metric $ds_\epsilon^2$ and area element  $d\sigma^\epsilon$ for $T^2_\epsilon$ are given by
\begin{equation*}
\begin{cases}
\ \ ds_\epsilon^2=\sum_{i,j=1}^{2} g^\epsilon_{ij}dx^idx^j=r_\epsilon^2(\theta) d\varphi^2+ \left[ (R+r_\epsilon(\theta)\cos \varphi)^2+(r_\epsilon^\prime(\theta))^2 \right] d\theta^2,\\
\ \ d\sigma^\epsilon=\sqrt{|g^\epsilon|} d\varphi d\theta=r_\epsilon(\theta) \sqrt{ (R+r_\epsilon(\theta)\cos \varphi)^2+(r_\epsilon^\prime(\theta))^2 }   d\varphi d\theta.
\end{cases} 
\end{equation*}
The Riemannian gradient $\Delta_{g^\epsilon}u$ of $u$ with respect to $g^\epsilon$ on $T^2_\epsilon$ 
is given by
 \begin{equation}\label{gradient}
\nabla_{g^\epsilon} u =
\begin{pmatrix}
\displaystyle \dfrac{1}{r_\epsilon^2 (\theta)} \partial_\varphi u \\[0.5cm]
\dfrac{1}{(R+r_\epsilon(\theta)\cos \varphi)^2+(r_\epsilon^\prime(\theta))^2}  \partial_\theta u
\end{pmatrix},
\end{equation}
and the Laplace-Beltrami operator $\Delta_{g^\epsilon}$ on $T^*$ is expressed as
\begin{equation}\label{LB}
\Delta_{g^\epsilon} u=  \frac 1{r_\epsilon^2(\theta)}u_{\varphi\varphi} +\frac 1{\Phi^2}u_{\theta\theta} + \frac{\Phi_\varphi}{r_\epsilon^2(\theta) \Phi}u_\varphi+\frac{r_\epsilon^\prime(\theta)\Phi-r_\epsilon(\theta)\Phi_\theta}{r_\epsilon(\theta)\Phi^3}u_\theta,
\end{equation}
where we set $\Phi = \Phi(\varphi, \theta) = \sqrt{ (R+r_\epsilon(\theta)\cos \varphi)^2+(r_\epsilon^\prime(\theta))^2 }$.
%

\subsection{Existence of patterns}
The arguments in this section follow those used in \cite[Section 4.1]{kamal}. We only show the main points of arguments in each proof. For further detail see \cite[Section 4.1]{kamal}. 
\begin{theorem}\label{exist}
	Let $U$ be the pattern of \eqref{eq1} for  $M=T^2$ given by {\rm Theorem \ref{t2}}. There exists $\epsilon_0>0$ such that  for each $|\epsilon| \in (0,\epsilon_0)$, a pattern $U^\epsilon$ of  \eqref{eq1}  for  $M=T^2_\epsilon$ exists.
\end{theorem}
\begin{lemma}\label{apl}
Let $0<\alpha<1$. There exists $\epsilon_1>0$ such that for each $|\epsilon| \in (0,\epsilon_1)$, a stationary solution $U^\epsilon$ of \eqref{eq1}  for $M=T^2_\epsilon$ and $\delta_1(\epsilon)> 0$ with $\lim\limits_{\epsilon\rightarrow 0} \delta_1(\epsilon)=0$  exist and satisfy
		\begin{equation*}
	\Vert U^\epsilon- U\Vert_{\mathnormal{C}^{2,\alpha}(I)} < \delta_1(\epsilon),\text{ if } |\epsilon| \in (0, \epsilon_1),
	\end{equation*}
	where $I=S^1\times S^1$ is given in \eqref{parametrization of perturbed tori}.
\end{lemma}
\begin{proof} Set
\begin{equation*}
X=\left(-\frac 12(R-r),\frac 12(R-r)\right)\subset \mathbb R\  \mbox{ and }\  Y=C^{2,\alpha}(I).
\end{equation*}
Let $F$ be a mapping  from $X \times Y $ to  $\mathnormal{C}^\alpha (I)$ defined by
\begin{equation*}
F(\epsilon,v)= \Delta_{g^\epsilon}(U+v)+f(U+ v)\ \mbox{ for }\ (\epsilon, v)\in X \times Y.
\end{equation*}
With the aid of the implicit function theorem, we solve $F=0$ near the point $(0,0)$. 
We notice that $F(0,0)=0$ and $F$ is of class $C^1$.
The partial Fr\'echet derivative of the mapping $F(\epsilon,v)$ with respect to $v$ at $(0,0)$ is expressed as
\begin{equation*}
\dfrac{\partial F}{\partial v}(0, 0)q=\Delta_g q+f^\prime(U)q \quad \mbox{ for } q \in Y.
\end{equation*} 
Let us show that $\dfrac{\partial F}{\partial v}(0, 0)$ is invertible. For each  $h \in C^\alpha(I)$, we consider the following problem for $q$:
\begin{equation}
\label{eqh}
\Delta_{g} q+f^\prime(U) q=h  \mbox{ in }T^2.
\end{equation}
Since $\lambda_1>0$, the standard theory of elliptic partial differential equations of second order provides us a unique solution $q\in Y$ of \eqref{eqh} and a constant $C>0$ independent of $q$ and $h$ satisfying 
 \begin{equation*}
 \left \Vert\left[\dfrac{\partial F}{\partial v}(0, 0)\right]^{-1}h\right\Vert_{C^{2,\alpha}(I)} =\Vert q\Vert_{C^{2,\alpha}(I)} \le C \Vert h\Vert_{C^{\alpha}(I)}
 \end{equation*}
 for each $h \in C^\alpha(I)$. Thus,  by the implicit function theorem (see \cite[Theorem 15.1, p. 148]{deimling} or \cite[Theorem 2.7.2, p.34]{nirenberg}), there exist $\epsilon_1\in \left(0, \frac 12(R-r)\right)$ with $\mathcal N=(-\epsilon_1,\epsilon_1)$ and a unique $C^1$ mapping $v:\mathcal N \to Y$ such that $v(0)=0$ and for every $\epsilon\in\mathcal N$
   \begin{equation}\label{ift}
  F(\epsilon,v(\epsilon))=\Delta_{g^\epsilon} (U^\epsilon )+f(U^\epsilon)=0,
  \end{equation}  
 where we set $U^\epsilon=U+v(\epsilon)$. This yields the conclusion.
 \end{proof}
 
Now, we are in position to prove the stability of the stationary solution $U^\epsilon$ of \eqref{eq1} for $M=T^2_\epsilon$.
Let $\lambda_1^\epsilon$ be the principal eigenvalue with the normalized eigenfunction $\phi_1^\epsilon \in H^1(I)$ of \eqref{egvals} linearized at $U^\epsilon$ for  $M=T^2_\epsilon$. Then we have
\begin{align}\label{pe}
\begin{cases}
\ \ \Delta_{g^\epsilon} \phi_1^\epsilon+f^\prime(U^\epsilon) \phi_1^\epsilon=-\lambda_1^\epsilon\phi_1^\epsilon\ \mbox{ in }\ T^2_\epsilon,\\
\ \ \|\phi_1^\epsilon\|_{L^2(T^2_\epsilon)}=1, \quad \phi_1^\epsilon>0\ \mbox{ in }\ T^2_\epsilon, 
\end{cases}
\end{align}
where
\begin{equation}\label{pe-ev} 
\lambda_1^\epsilon=  \inf_{\substack{\phi \neq 0\\\phi \in H^1(T^2_\epsilon)}}  \dfrac{ \int\limits_{T^2_\epsilon} \left( |\nabla_{g^\epsilon} \phi|^2 - f^\prime (U^\epsilon)\phi^2 \right) d\sigma^\epsilon}{\int\limits _{T^2_\epsilon} \phi^2 d\sigma^\epsilon} \left(=  \int\limits_{T^2_\epsilon} \left[ |\nabla_{g^\epsilon} \phi_1^\epsilon|^2 - f^\prime (U^\epsilon)(\phi_1^\epsilon)^2 \right] d\sigma^\epsilon\right).
\end{equation}

\begin{lemma}\label{int}
	For every $| \epsilon|  \in (0,\epsilon_1)$, there exists $\delta_2(\epsilon)>0$ with $\lim\limits_{\epsilon\rightarrow 0}\delta_2(\epsilon)=0$ such that
	\begin{enumerate}
		\item[{\rm (i)}]
		$\| f^\prime (U)- f^\prime (U^\epsilon)\|_{\infty} \le \delta_2 (\epsilon)$,
		\item[{\rm (ii)}] $(1 -\delta_2(\epsilon))d\sigma  \le  d\sigma^\epsilon\le  (1 +\delta_2(\epsilon))d\sigma$,
		\item[{\rm (iii)}] $(1- \delta_2(\epsilon) ) |\nabla_{g}\phi_1^\epsilon|^2 \le |\nabla_{g^\epsilon}\phi_1^\epsilon|^2 \le (1+ \delta_2(\epsilon) ) |\nabla_{g}\phi_1^\epsilon|^2 $.
	\end{enumerate}
\end{lemma}
\begin{proof}
	Assertion (i) comes from Lemma \ref{apl} and the continuity of $f^\prime$ for some $\delta_2(\epsilon)>0$ with $\lim\limits_{\epsilon\rightarrow 0}\delta_2(\epsilon)=0$.
Next, represent $\sqrt{|g^\epsilon|}$ in assertion (ii) and $|\nabla_{g^\epsilon}\phi_1^\epsilon|^2$ in assertion (iii) by using Taylor expansion with respect to $\epsilon$ at $\epsilon=0$. By the continuity of $\sqrt{|g^\epsilon|}$ and $|\nabla_{g^\epsilon}\phi_1^\epsilon|^2$, we can generate  a chain of inequalities on each expansion and choose $\delta_2(\epsilon)>0$ smaller to obtain both assertion (ii) and assertion (iii)
\end{proof}
\begin{lemma}\label{bounded}
	There exists a constant $C^* >0$ such that  if $ |\epsilon| <  \epsilon_1$ then
	$$
	\int\limits_{T^2_\epsilon}|\nabla_{g^\epsilon} \phi_1^\epsilon|^2 d\sigma^\epsilon \le C^*.
	$$
\end{lemma}
\begin{proof}
	From \eqref{pe}, we have that
	\begin{align}\label{apa}
	\int\limits_{T^2_\epsilon}  |\nabla_{g^\epsilon} \phi_1^\epsilon|^2  d\sigma^\epsilon=\lambda_1^\epsilon + \int\limits_{T^2_\epsilon} f^\prime (U^\epsilon) (\phi_1^\epsilon)^2 d\sigma^\epsilon
	\end{align} 
First we need to prove that $\lambda_1^\epsilon$ is bounded from above. By \eqref{pe-ev}, every $\phi\in H^1(I)$ satisfies
	\begin{align*}
	\lambda_1^\epsilon\le \dfrac{\int\limits_{T^2_\epsilon} \left( |\nabla_{g^\epsilon} \phi|^2  - f^\prime (U^\epsilon) \phi^2\right)  d\sigma^\epsilon }{ \int\limits_{T^2_\epsilon} \phi^2 d\sigma^\epsilon}.
	\end{align*}
	Choose $\phi\equiv1$ and use assertion (i) to obtain
	\begin{align*}
\lambda_1^\epsilon \le \delta_2(\epsilon)+\max_{T^2} |f^\prime(U)|.
	\end{align*}
Since $\|\phi_1^\epsilon\|_{L^2(T^2_\epsilon)}=1$, assertion (i) gives	
	\begin{align*}
	\int\limits_{T^2_\epsilon} f^\prime (U^\epsilon) (\phi_1^\epsilon)^2 d\sigma^\epsilon &= \int\limits_{T^2_\epsilon} \left(f^\prime (U^\epsilon)-f^\prime(U)\right) (\phi_1^\epsilon)^2 d\sigma^\epsilon+  \int\limits_{T^2_\epsilon}f^\prime (U) (\phi_1^\epsilon)^2 d\sigma^\epsilon \\[0.25cm]
	&\le \delta_2(\epsilon) + \max_{T^2} |f^\prime (U)|.
	\end{align*}
	Then, \eqref{apa} yields the conclusion.
\end{proof}

\begin{lemma}\label{uni}
	$\lambda_1^\epsilon \rightarrow \lambda_1 \textrm{ as }  \epsilon \rightarrow 0$.
\end{lemma}
\begin{proof} This follows directly from the same argument as in the proof of \cite[Lemma 4.5, pp. 11--12 ]{kamal}.
From \eqref{ev}, we have
\begin{align*}
\lambda_1\le \dfrac{\int\limits_{T^2} \left( |\nabla_{g} \phi_1^\epsilon|^2  - f^\prime (U) (\phi_1^\epsilon)^2\right)  d\sigma}{ \int\limits\limits_{T^2} (\phi_1^\epsilon)^2 d\sigma}.
\end{align*}
With the aid of Lemmas \ref{int} and \ref{bounded}, we infer that there exists $\delta_3(\epsilon)>0$ with $\lim\limits_{\epsilon\rightarrow 0}\delta_3(\epsilon)=0$ satisfying
\begin{equation}\label{1st estimate}
\lambda_1 \le \lambda_1^\epsilon +\delta_3(\epsilon).
\end{equation}
By proceeding similarily, from \eqref{pe-ev}, we have
\begin{align*}
\lambda^\epsilon_1\le \dfrac{\int\limits_{T_\epsilon^2} \left( |\nabla_{g^\epsilon} \phi_1|^2  - f^\prime (U^\epsilon) (\phi_1)^2\right)  d\sigma^\epsilon}{ \int\limits\limits_{T_\epsilon^2} (\phi_1)^2 d\sigma^\epsilon},
\end{align*}
and hence we infer that there exists $\delta_4(\epsilon)>0$ with $\lim\limits_{\epsilon\rightarrow 0}\delta_4(\epsilon)=0$ satisfying
\begin{equation}\label{2nd estimate}
\lambda_1^\epsilon \le \lambda_1 +\delta_4(\epsilon).
\end{equation}
Combining \eqref{1st estimate} and \eqref{2nd estimate} yields the conclusion.
\end{proof}

\begin{proof}[Proof of  Theorem \ref{exist}]
From Lemma \ref{apl}, we conclude that the stationary solution $U^\epsilon$ exists in the neighborhood of $\epsilon=0$.  Then, by Lemma \ref{uni}, the principal eigenvalue $\lambda_1^\epsilon $ is positive  for sufficiently small  $|\epsilon|$. This completes the proof.
\end{proof}

\subsection{Proof of Theorem \ref{mt}}
First of all, we mention the symmetry of $U^\epsilon$ coming from the fact that $U^\epsilon$ is uniquely determined in a neighborhood of $U$ by the implicit function theorem.
To be precise, the implicit function theorem together with the symmetry of both $U$ and $T_\epsilon^2$ gives us the symmetry of $U^\epsilon$ with respect to $T_\epsilon^2\cap H$ for the following $n+1$ planes $H$:
\begin{equation}
\label{hyperplanes for symmetry}
H=\{ x_3=0\}, \{-x_1\sin\theta_k+x_2\cos\theta_k=0\}\ \mbox{ with }  k=0,1, \dots, n-1,
\end{equation}
where  $\theta_k=\frac {2k+1}{2n}\pi$.
 Thus we have in particular that if $ |\epsilon| < \epsilon_0$ then
\begin{align}
&\dfrac{\partial U^\epsilon}{\partial \varphi}=0\ \mbox{ for every }\  (\varphi,\theta)\in \{0,\pi\} \times S^1,\label{partial derivative in varphi}\\
&\dfrac{\partial U^\epsilon}{\partial \theta}=0\ \mbox{ for every }\  (\varphi,\theta)\in S^1\times \{\theta_k\, |\,  k=0,1, \dots, 2n-1\}.\label{partial derivative in theta}
\end{align}
These imply that if $ |\epsilon| < \epsilon_0$, then $U^\epsilon$ has at least $4n$ critical points in $T^2_\epsilon$ corresponding to the points in the following finite set $\mathcal C$
\begin{equation}
\label{4n critical points}
\mathcal C=\{0,\pi\}\times\{\theta_k\, |\,  k=0,1, \dots, 2n-1\}. 
\end{equation}	
In view of \eqref{Laplace operator for torus},  we recall that $U^0=U=U(\varphi)$ satisfies 
 \begin{equation}\label{ode for U}
\frac{1}{r^2} U_{\varphi\varphi} - \frac{\sin \varphi}{r(R+r \cos \varphi)} U_\varphi+f(U)=0\ \mbox{ in } T^2.
\end{equation}
Hence we have from \eqref{partial derivative in varphi}
 $$
\frac{1}{r^2} U_{\varphi\varphi} +f(U)=0\ \mbox{ at }  \varphi=0, \pi.
$$
Therefore, since $U$ is nonconstant and achieves its maximum at  $\varphi=\pi$ and its minimum at $\varphi=0$,  from the unique solvability of the Cauchy problem for the ordinary differential equation \eqref{ode for U} we must have that 
\begin{equation}
\label{2nd derivatives}
f(U(0)) < 0 < f(U(\pi)) \ \mbox{ and hence }\ U_{\varphi\varphi}(0)>0>U_{\varphi\varphi}(\pi).
\end{equation}
Since the set of critical points of $U$ equals $T^2\cap\{x_3=0\}$, it follows from Lemma \ref{apl}, \eqref{partial derivative in varphi} and \eqref{2nd derivatives} that there exists $\tau_1\in (0,\epsilon_0)$ satisfying
that if $|\epsilon| < \tau_1$, the set of critical points of $U^\epsilon$ is contained in $T_\epsilon^2\cap \{x_3=0\}$.
Thus, in order to determine all the critical points of $U^\epsilon$, we need to examine whether the derivative of $U^\epsilon$ with respect to $\theta$ vanishes or not for $\varphi=0,\pi$. 

Observe that as  $\epsilon \to 0$,
\begin{equation}\label{V1}
U^\epsilon=U+v(\epsilon)=U + \epsilon \dfrac{\partial U^\epsilon}{\partial \epsilon}\bigg|_{\epsilon=0}+ o(\epsilon).
\end{equation} 
Set $V=\dfrac{\partial U^\epsilon}{\partial \epsilon}\bigg|_{\epsilon=0}$.  By differentiating \eqref{V1} with respect to $\theta$ twice, we see that as  $\epsilon \rightarrow 0$
\begin{equation}\label{du}
\dfrac{\partial U^\epsilon}{\partial \theta}= \epsilon V_\theta+ o(\epsilon)\  \mbox{ and }\  \dfrac{\partial^2 U^\epsilon}{\partial \theta^2}= \epsilon V_{\theta\theta}+ o(\epsilon).
\end{equation}
We will examine $V$ and evaluate $ V_\theta, V_{\theta\theta}$ for $\varphi=0,\pi$.
 Recall that 
\begin{equation}\label{ff}
F(\epsilon, v(\epsilon) )= \Delta_{g^\epsilon} U^\epsilon+ f(U^\epsilon)=0 \ \mbox{ for  every } (\varphi,\theta)\in I=S^1\times S^1.
\end{equation}
Differentiating \eqref{ff} with respect to $\epsilon$ yields that
\begin{align*}
0=\frac {\partial}{\partial\epsilon}\left(F(\epsilon, v(\epsilon) )\right)=& \frac 1{r_\epsilon^2}U^\epsilon_{\varphi\varphi\epsilon} +\frac {\partial}{\partial\epsilon}\left(\frac 1{r_\epsilon^2}\right) U^\epsilon_{\varphi\varphi} +\frac 1{\Phi^2}U^\epsilon_{\theta\theta\epsilon}+\frac {\partial}{\partial\epsilon}\left( \frac 1{\Phi^2}\right) U^\epsilon_{\theta\theta} \\
&+\frac{\Phi_\varphi}{r_\epsilon^2 \Phi}U^\epsilon_{\varphi\epsilon}+\frac {\partial}{\partial\epsilon}\left(\frac{\Phi_\varphi}{r_\epsilon^2 \Phi}\right) U^\epsilon_\varphi\\
&+\frac{r_\epsilon^\prime\Phi-r_\epsilon\Phi_\theta}{r_\epsilon\Phi^3}U^\epsilon_{\theta\epsilon}+\frac {\partial}{\partial\epsilon}\left(\frac{r_\epsilon^\prime\Phi-r_\epsilon\Phi_\theta}{r_\epsilon\Phi^3}\right) U^\epsilon_\theta+f^\prime(U^\epsilon)U^\epsilon_\epsilon.
\end{align*}
Since $U^0$ depends only on $\varphi$, by setting $\epsilon=0$, we have
\begin{equation}\label{vv}
 \Delta_{g} V + f^\prime(U)V=\frac 2r \sin (n\theta) \left[-f(U)+ \dfrac{R\sin \varphi}{2r(R+r\cos \varphi)^2} U_{\varphi} \right].
\end{equation}
Observe that the right-hand side of \eqref{vv} is infinitely differentiable  in $\theta$ and all the coefficients of the left-hand side of \eqref{vv} are independent of $\theta$.  Then, by the standard regularity theory for elliptic partial differential equations (see \cite{Gilbarg-Trudinger}), we may
differentiate \eqref{vv} with respect to $\theta$ twice to obtain
\begin{equation}\label{egvalv}
 \Delta_{g} \left(\dfrac{V_{\theta\theta}}{n^2}+V \right)+ f^\prime(U)\left(\dfrac{V_{\theta\theta}}{n^2}+V \right)=0.
\end{equation}
Then the function $\dfrac{V_{\theta\theta}}{n^2}+V$  might be an eigenfunction of problem \eqref{egvals} linearized at $U$ for $M=T^2$ which corresponds to eigenvalue $0$. Since the principal eigenvalue $\lambda_1$   is  positive, $\dfrac{V_{\theta\theta}}{n^2}+V$ must vanish identically.  Thus, it is easy to express $V$ as 
\begin{equation}\label{vsl}
V(\varphi,\theta)=C_1(\varphi) \cos (n\theta)+C_2(\varphi)\sin (n\theta)\ \mbox{ for every } (\varphi,\theta)\in I
\end{equation}
for some  functions $C_1(\varphi), C_2(\varphi)$ of class $C^2$.
It remains to examine  $C_1(\varphi)$ and $C_2(\varphi)$. 

Subtituting \eqref{vsl} into \eqref{vv} yields  the following two ordinary differential equations
\begin{align}
&C^{\prime \prime}_1(\varphi) -\dfrac{r\sin \varphi}{R+r\cos \varphi}C^\prime_1(\varphi)- B(\varphi) C_1(\varphi)=0,\label{ode1}\\
&C^{\prime \prime}_2(\varphi) -\dfrac{r\sin \varphi}{R+r\cos \varphi}C^\prime_2(\varphi)-B(\varphi) C_2(\varphi)=A(\varphi),\label{ode2}
\end{align}
where we set
$$
A(\varphi)=2r\left[-f(U)+ \dfrac{R\sin \varphi}{2r(R+r\cos \varphi)^2} U_{\varphi} \right] \ \mbox{ and }\ B(\varphi)=r^2\left\lbrack \dfrac{n^2}{(R+r\cos \varphi)^2} -f^\prime (U)\right\rbrack.
$$
We choose the number $N \in \mathbb N$ in Theorem \ref{mt} as
\begin{equation}
\label{large n}
\displaystyle N^2> \max_{\varphi \in S^1} |f^\prime (U)| (R+r)^2.
\end{equation}
 Let us assume that $n\ge N$ from now on. 
Then $B(\varphi)$ is positive everywhere. This fact together with \eqref{ode1} yields that $C_1(\varphi)\equiv 0$, since the maximum principle implies that $C_1$ achieves neither its positive maximum nor its negative minimum. Therefore, under \eqref{large n} we have from \eqref{vsl}
\begin{equation}\label{vsl2}
V(\varphi,\theta)=C_2(\varphi)\sin (n\theta)\ \mbox{ for every } (\varphi,\theta)\in I.
\end{equation}
From \eqref{partial derivative in varphi}, we have that $U^\epsilon_\varphi\big|_{\varphi=0,\pi}=\epsilon V_\varphi\big|_{\varphi=0,\pi}+o(\epsilon)=0.$ Hence, 
\begin{equation*}
0=V_\varphi (\varphi,\theta)=C_2^\prime(\varphi)\sin (n\theta)\  \mbox{ for every } (\varphi,\theta)  \in \{0,\pi\} \times S^1.
\end{equation*}
Then we obtain 
\begin{equation}\label{boundary}
 C_2^\prime(0)= C_2^\prime(\pi)=0.
\end{equation}

Let us first show that $C_2(0)\not=0$. Multiplying \eqref{ode2} by $R+r\cos\varphi$ yields that
\begin{equation}
\label{pre 2nd order ode for C_2}
\left[(R+r\cos\varphi)C_2^\prime(\varphi)\right]^\prime=(R+r\cos\varphi)\left[B(\varphi)C_2(\varphi)+A(\varphi)\right].
\end{equation}
Then, with the aid of \eqref{boundary}, by integrating \eqref{pre 2nd order ode for C_2}
 in $\varphi$ from $0$ to $\pi$, we have
\begin{equation}\label{integral zero to pi}
\int_0^\pi(R+r\cos\varphi)\left[B(\varphi)C_2(\varphi)+A(\varphi)\right]d\varphi=0.
\end{equation}
Multiplying \eqref{ode for U} by $r^2(R+r\cos\varphi)$ yields that
\begin{equation}
\label{pre 2nd order ode for U}
\left[(R+r\cos\varphi)U^\prime(\varphi)\right]^\prime=-r^2(R+r\cos\varphi)f(U).
\end{equation}
Since $U^\prime(0)=U^\prime(\pi)=0$, by integrating \eqref{pre 2nd order ode for U}
 in $\varphi$ from $0$ to $\pi$, we have
\begin{equation}\label{integral zero to pi for f}
\int_0^\pi(R+r\cos\varphi)f(U)d\varphi=0.
\end{equation}
Then \eqref{integral zero to pi for f} yields that
$$
\int_0^\pi (R+r\cos\varphi)A(\varphi)d\varphi=\int_0^\pi\dfrac{R\sin \varphi}{R+r\cos \varphi} U^\prime d\varphi > 0,
$$
where we used that $U^\prime > 0$ in $(0,\pi)$, and hence by \eqref{integral zero to pi} we conclude that
\begin{equation}
\label{negativity of the important integral}
\int_0^\pi(R+r\cos\varphi)B(\varphi)C_2(\varphi)d\varphi<0.
\end{equation}
Suppose that $C_2(0)=0$.  Since $C^\prime_2(0)=0$, substituting $\varphi=0$ into \eqref{ode2} yields that
$$
C_2^{\prime\prime}(0)=A(0) = -2r f(U(0)) > 0,
$$
where we used \eqref{2nd derivatives}. Therefore there exists $\delta\in (0,\pi)$ satisfying
$$
C_2(\varphi) > 0 \mbox{ and } C_2^\prime(\varphi) > 0\ \mbox{ for every } \varphi \in (0,\delta).
$$
On the other hand, it follows from \eqref{negativity of the important integral} and the positivity of $B(\varphi)$ that $C_2(\varphi)$ must be negative at some point in $(0,\pi)$.
Thus we may find $\delta_*\in [\delta, \pi)$ such that
\begin{equation}
\label{positivity of C_2}
C_2^\prime(\delta_*)=0\ \mbox{ and } C_2(\varphi) > 0\ \mbox{ for every } \varphi \in (0,\delta_*).
\end{equation}
Then, replacing $\pi$ by $\delta_*$ and using the same arguments as in getting \eqref{integral zero to pi} and \eqref{integral zero to pi for f} yield that
$$
\int_0^{\delta_*}(R+r\cos\varphi)\left[B(\varphi)C_2(\varphi)+A(\varphi)\right]d\varphi=\int_0^{\delta_*}(R+r\cos\varphi)f(U)d\varphi=0.
$$
Also, by combining these and using that $U^\prime > 0$ in $(0,\pi)$, we arrive at
$$
\int_0^{\delta_*}(R+r\cos\varphi)B(\varphi)C_2(\varphi)d\varphi<0,
$$
which contradicts \eqref{positivity of C_2} and the positivity of $B(\varphi)$.  Thus this concludes that $C_2(0)\not=0$. Since from \eqref{boundary} $ C_2^\prime(\pi)= C_2^\prime(2\pi)=0$,
by employing the same argument on the interval $(\pi, 2\pi)$ as on $(0,\pi)$, we may also have that $C_2(\pi)\not=0$. Hence it follows from \eqref{vsl2} that for $(\varphi,\theta)\in \{0,\pi\}\times [0,2\pi)$
$$
V_\theta\not=0\ \mbox{ if }\ \theta\not\in \{\theta_k\, |\,  k=0,1, \dots, 2n-1\} \ \mbox{ and }\  V_{\theta\theta}\not=0\ \mbox{ if }\ \theta\in\{\theta_k\, |\,  k=0,1, \dots, 2n-1\},
$$
where  $\theta_k=\frac {2k+1}{2n}\pi$ are given in \eqref{hyperplanes for symmetry}.
Therefore, by combining this with \eqref{partial derivative in theta} and \eqref{du},  we find $\tau_2\in (0,\tau_1)$ such
that if $|\epsilon| < \tau_2$, the set of critical points of $U^\epsilon$ correspond to $\mathcal C$ given by \eqref{4n critical points},
which consists of exactly $4n$ points in $T^2_\epsilon$. This completes the proof of Theorem \ref{mt}.



\begin{small}
	 
\end{small}
\end{document}